\documentclass[11pt, a4paper]{amsart}
\usepackage{amsmath}
\usepackage{geometry,amsthm,graphics,tabularx,amssymb,shapepar, braket}
\usepackage{latexsym,amsfonts,amssymb,amsmath,amsxtra,color,mathrsfs}
\usepackage{amscd}
\usepackage{mathtools}
\usepackage{enumitem}
\allowdisplaybreaks[4]

\usepackage{hyperref}
\usepackage[nameinlink,capitalise,noabbrev]{cleveref}
\usepackage{pdfsync}

\usepackage{verbatim}

\newcommand{\BC}{{\mathbb {C}}}

\newcommand{\BF}{{\mathbb {F}}}

\newcommand{\CK}{{\mathcal {K}}}

\newcommand{\CW}{{\mathcal {W}}}

\newcommand{\GL}{{\mathrm{GL}}}

\newcommand{\Hom}{{\mathrm{Hom}}}

\renewcommand{\Re}{{\mathrm{Re}}}

\newcommand{\SL}{{\mathrm{SL}}}

\newcommand{\p}{\frac{\partial\bar\partial}{\Pi i}}

\newcommand{\bs}{\backslash}

\newcommand{\diag}{\operatorname{diag}}

\newcommand{\od}{\operatorname{d}}

\renewcommand{\p}{\mathfrak p}

\renewcommand{\o}{\mathfrak o}

\newcommand{\C}{\mathbb{C}}

\newcommand{\ve}{{\vee}}

\newcommand{\abs}[1]{\lvert#1\rvert}

\newcommand{\be}{\begin {equation}}
\newcommand{\ee}{\end {equation}}
\newcommand{\bee}{\begin {equation*}}
\newcommand{\eee}{\end {equation*}}

\newcommand{\St}{\operatorname{St}}

\setlist{nosep,leftmargin=2em}

\theoremstyle{plain}

\theoremstyle{plain}

\newtheorem*{theoremA'}{Theorem A'}

\theoremstyle{plain}
\newtheorem{lem}{Lemma}[section]

\newtheorem{thml}[lem]{Theorem}
\newtheorem{leml}[lem]{Lemma}

\theoremstyle{plain}

\theoremstyle{remark}

\newtheorem{remarkl}[lem]{Remark}

\theoremstyle{definition}

\theoremstyle{plain}
\newtheorem{theorem}{Theorem}[section]

\newtheorem{corollary}[theorem]{Corollary}

\numberwithin{equation}{section}

\begin{document}

\title[Rankin--Selberg integrals of opposite conductor-one  newforms]{Rankin--Selberg integrals of opposite conductor--one newforms}

 \author[D. Liu]{Dongwen Liu}
\address{School of Mathematical Sciences, Zhejiang University, 866 Yuhangtang Road, Hangzhou 310058, P.R. China}
\email{maliu@zju.edu.cn}

\author[L. Zhang]{Lei Zhang}
\address{Department of Mathematics, National University of Singapore, Singapore 119076}
\email{matzhlei@nus.edu.sg}

	%\date{\today}
	%\subjclass[2020]{22E50, 43A80}
	%\keywords{Newform, Whittaker functions, Rankin--Selberg integral, Hecke algebra}

\begin{abstract}
Let $F$ be a nonarchimedean local field of characteristic zero and let $n\geq2$.  For
$r=n,n+1$, let $\Pi_r$ be an irreducible tempered representation of
$\GL_r(F)$ of conductor one and with trivial central character.  
%The  long Weyl element carries the usual one-dimensional newform space in $\Pi_r$ to a one-dimensional space $\sigma_r$ fixed by an opposite parahoric subgroup.  
We evaluate the Rankin--Selberg integral of {\it opposite} newforms in $\Pi_{n+1}\times  \Pi_n$ explicitly and show that its central value is nonzero. As an application, this implies a case of Disegni--Zhang's conjecture on the nonvanishing of local relative characters.
\end{abstract}

\maketitle

\section{Introduction}
\label{sec:intro}

Let $F$ be a nonarchimedean local field of characteristic zero, with normalized absolute value $\abs{\cdot}_F$.  Let $\o$ be the ring of integers of $F$, with maximal ideal $\p$, a uniformizer $\varpi$, and 
residue field $\o/\p\cong\mathbb F_q$.
For $r\geq1$, write
\[
 G_r :=\GL_r(F),\qquad K_r:=\GL_r(\o).
\]
For every integer $c\geq 0$, define the congruence subgroup
\[
K_{c, r}:=\{ k\in K_r: (0,\ldots,0,1)k\equiv(0,\ldots,0,1)\pmod{\p^c}\}.
\]
The conductor $c(\pi)$ of an irreducible generic representation $\pi$ of
$G_r$ is the smallest integer $c\geq0$ for which
$\pi^{K_{c,r}}\neq0$.  Then
$\dim\pi^{K_{c(\pi),r}}=1$, and a basis vector of this space is called a
newform of $\pi$.

Let $\Pi_n$ and $\Pi_{n+1}$ be irreducible generic representations of
$G_n$ and $G_{n+1}$, respectively.  We use the diagonal embedding
\[
G_n \hookrightarrow G_{n+1}\times G_n,\qquad g \mapsto \left(\begin{pmatrix} g  \\ & 1\end{pmatrix}, g \right).
\]
If $\Pi_n$ is unramified, that is, if $c(\Pi_n)=0$, the 
Rankin--Selberg integral (\cite{JPSS}) of the Whittaker newforms of $\Pi_{n+1}$ and
$\Pi_n$ equals a constant multiple of the Rankin--Selberg $L$-function.  We refer to
\cite{JPSS81,J,Matringe,Miyauchi} for the theory of Whittaker
newforms.

Suppose instead that $c(\Pi_n)\geq1$.  By the multiplicity-one theorem
\cite{AGRS} and the theory of Rankin--Selberg integrals, we have
\[
\dim\Hom_{G_n}(\Pi_{n+1}\otimes\Pi_n,\C) =1,
\]
where $G_n$ acts through
the diagonal embedding above.  Let
$\ell\in\Hom_{G_n}(\Pi_{n+1}\otimes\Pi_n,\C)$, and let $v_r\in\Pi_r$ be a newform.
Then $v\mapsto\ell(v_{n+1}\otimes v)$ is a $K_n$-invariant linear
functional on $\Pi_n$.  Since $c(\Pi_n^\vee) = c(\Pi_n)>0$,  we have $(\Pi_n^\vee)^{K_n}=\{0\}$.  It follows that
\be \label{van}
\ell(v_{n+1}\otimes v_n)=0.
\ee
Thus the two usual newforms do not form  test vectors for the
$G_n$-invariant linear functional $\ell$.

From now on, for $r=n,n+1$, assume that $\Pi_r$ is tempered, has
conductor one, and has trivial central character.  Define the opposite
congruence subgroup
\[
 K_{1,r}^-:= w_r K_{1,r} w_r^{-1}
 =\{k\in K_r:(1,0,\ldots,0)k\equiv(1,0,\ldots,0)\pmod\p\}.
\]
Here and henceforth, $w_r\in G_r$ denotes the anti-diagonal permutation matrix.  
%We also define the opposite parahoric subgroup
%\[
%\CK_{1,r}^-:=\o^\times \cdot K_{1,r}^-,
%\]
%where $\o^\times$ is identified with the maximal compact subgroup of the center of $G_r$. 
The action by $w_r$ gives a linear isomorphism of one-dimensional spaces
\[
\Pi_r^{K_{1,r}} \xrightarrow{\sim} \Pi_r^{K_{1,r}^-}=:\sigma_r.
\]
Theorem \ref{thm:main} below implies that, for every nonzero
$\ell\in\Hom_{G_n}(\Pi_{n+1}\otimes\Pi_n,\C)$,
\be \label{nonvanish}
\ell|_{\sigma_{n+1}\otimes \sigma_n} \neq 0.
\ee

Let $N_r$ be the upper triangular maximal unipotent subgroup of
$G_r$.  Fix a nontrivial unitary additive character
$\psi:F\to\C^\times$ with conductor $\o$, and define
\[
 \psi_r(x):=\psi\!\left((-1)^r\sum_{i=1}^{r-1}x_{i,i+1}\right),
 \qquad x\in N_r.
\]
Write $\CW(\Pi_r,\psi_r)$ for the Whittaker model of $\Pi_r$ with respect to
$(N_r,\psi_r)$.

For $r\geq 1$, let $B_r=N_rA_r$ be the upper triangular Borel subgroup
of $G_r$, with diagonal torus $A_r \cong (F^\times)^r$. Denote by $\delta_r$ the modular character of $B_r$.  We give $K_r$ probability measure and give $F^\times$ the
multiplicative Haar measure for which $\o^\times$ has volume one.  If
$\od\! a$ denotes the resulting product measure on $A_r$, the
quotient measure on $N_r\backslash G_r$ is
\[
\od\!  g= \delta_r^{-1}(a)\od \!a\od\!k,
\qquad g=uak,\quad u\in N_r,\ a\in A_r,\ k\in K_r.
\]
For $W_r\in\CW(\Pi_r,\psi_r)$, $r=n,n+1$, define the Rankin--Selberg integral
\cite{JPSS}
\be \label{RS}
Z(s, W_{n+1}, W_n):=\int_{N_n\bs G_n} W_{n+1}\begin{pmatrix} g  \\ & 1\end{pmatrix} W_n(g)\abs{\det g}_F^{s-\frac{1}{2}}\od\! g,\quad s\in \C,
\ee
which converges absolutely for $\Re(s)$ sufficiently large and defines a rational function of $q^s$.  Temperedness implies absolute
convergence at $s=\frac12$, and the bilinear map
$(W_{n+1},W_n)\mapsto Z(\frac12,W_{n+1},W_n)$ defines a nonzero element of
$\Hom_{G_n}(\Pi_{n+1}\otimes\Pi_n,\C)$.

By Lemma \ref{prop:class}, we have $n\geq2$ and the $L$-parameter of $\Pi_{n+1}$ (a Weil--Deligne representation) is of the form 
\be \label{pin+1}
\phi_{\Pi_{n+1}} = \St_2\otimes\chi  + \bigoplus^{n-1}_{i=1}\chi_i,
\ee
where $\chi$ and $\chi_i$ are  unitary unramified characters such that $\chi^2\chi_1\cdots\chi_{n-1}=1$, and $\St_2$ is the Weil--Deligne representation corresponding to the Steinberg
representation of $G_2$.  The representation $\Pi_n$ has the analogous $L$-parameter
\eqref{pin} below.  For $r=n,n+1$, let
$W_r^-\in\CW(\Pi_r,\psi_r)^{K_{1,r}^-}$ be the opposite Whittaker
newform normalized by
\[
W_r^-(w_r)=1.
\]

The main result of this paper is as follows. 

\begin{theorem}
\label{thm:main}
Let the notation and assumptions be as above.  Then
\[
 Z(s,W_{n+1}^-,W_n^-)
 =-\frac{\chi(\varpi) q^{ns- \frac{1}{2}}}{ \# \mathbb{P}^{n-1}(\mathbb{F}_q)}
L(s,\Pi_{n+1}\times\Pi_n)
\]
as a rational function of $q^s$. In particular, $Z(\frac{1}{2}, W_{n+1}^-, W_n^-)\neq 0$. 
\end{theorem}

Theorem \ref{thm:main} has the following application to the nonvanishing of local relative characters, which improves \cite[Theorem C]{DZ}. More details will be explained in Section \ref{sec:conj}.

\begin{corollary}\label{cor}
\cite[Conjecture 3.6.10]{DZ} holds when $\Pi_{n,v}$ and $\Pi_{n+1,v}$ both have conductor one.
\end{corollary}

This paper is organized as follows.  In \S\ref{sec2} we collect some preliminaries. We prove Theorem \ref{thm:main} by applying Hecke operators to the Rankin--Selberg integrals of Whittaker newforms 
and opposite Whittaker newforms in \S\ref{sec3} and \S\ref{sec4}, respectively. In \S\ref{sec:conj} we explain Corollary \ref{cor} and its proof.

\section{Preliminaries} \label{sec2}

In this section, we introduce some notation and recall a few known results that will be used in  the proof of Theorem \ref{thm:main}.

The following classification is given by \cite[Lemma 3.6.1]{DZ}.

\begin{leml}\label{prop:class} 
Assume that $\Pi_n$ is an irreducible tempered representation of
$G_n$ of conductor one and with trivial central character.  Then
$n\geq2$ and the $L$-parameter of $\Pi_n$ is 
\be \label{pin}
\phi_{\Pi_n} =  \St_2\otimes \xi + \bigoplus^{n-2}_{i=1} \xi_i 
\ee
where $\xi$ and $\xi_i$ are unitary unramified characters such that
$\xi^2\xi_1\cdots\xi_{n-2}=1$, and $\St_2$ is the Weil--Deligne representation corresponding to the Steinberg
representation of $G_2$.  
\end{leml}

In view of  Lemma \ref{prop:class}, we henceforth assume that
$\Pi_{n+1}$ and $\Pi_n$ have $L$-parameters as in \eqref{pin+1} and \eqref{pin},
respectively. Let 
\be \label{unr}
\phi_{\Pi_{n+1}^u} := \abs{\cdot}_F^{1/2} \chi + \bigoplus^{n-1}_{i=1} \chi_i,
\ee
which is the $L$-parameter of an unramified representation $\Pi_{n+1}^u$ of $G_n$ such that 
\[
L(s, \Pi_{n+1}^u) = L(s, \Pi_{n+1}).
\]
Similarly, we have an unramified representation $\Pi_{n}^u$ of $G_{n-1}$ with $L$-parameter 
\[
\phi_{\Pi_n^u} =  \abs{\cdot}_F^{1/2} \xi + \bigoplus^{n-2}_{i=1} \xi_i
\]
such that  $L(s, \Pi_n^u) = L(s,\Pi_n)$.  In view of the Clebsch-Gordan decomposition 
\[
S_2\otimes S_2 \cong S_3 \oplus S_1,
\]
where $S_m$ denotes the $m$-dimensional irreducible algebraic representation of $\SL_2(\BC)$, 
it is easy to see that
\be\label{naive}
L(s, \Pi_{n+1}^u \times \Pi_n^u ) = (1-q^{-s}\chi(\varpi)\xi(\varpi)) L(s, \Pi_{n+1} \times \Pi_n). 
\ee

For $r=n, n+1$, let $W_r^\circ\in \CW(\Pi_r, \psi_r)^{K_{1,r}}$ be the Whittaker newform normalized such that $W_r^\circ(1_r)=1$. Here and henceforth, $1_r$ denotes the $r\times r$ identity matrix. 
In \cite{BKL}, it is shown that $W_n^\circ$ and a unipotent average of $W_{n+1}^\circ$ form test vectors for $L(s, \Pi_{n+1}^u \times \Pi_n^u)$, which is recalled below.

For convenience, fix a lift $\omega: \mathbb{F}_q \hookrightarrow \frak o$ with $\omega(0)=0$, which in particular identifies $\mathbb{F}_q^\times$ with a subset of $\frak o^\times$. 
For a row vector $x\in F^r:=F^{1\times r}$, write
\[
u(x): = \begin{pmatrix} 1_r & x^{\rm t} \\ 0 & 1\end{pmatrix}\in N_{r+1},
\]
where $(\cdot)^{\rm t}$ denotes the transpose of a matrix. We have the following special case of \cite{BKL}.

\begin{thml} \label{BKL}
Let
\[
W_{n+1}^{\rm av} : = q^{1-n} \sum_{y\in \BF_q^{n-1}} u(\varpi^{-1}y, 1).W_{n+1}^\circ.
\]
Then 
\[
Z(s, W_{n+1}^{\rm av}, W_n^\circ) = \frac{1}{\#\mathbb{P}^{n-1}(\mathbb{F}_q)}L(s, \Pi_{n+1}^u\times \Pi_n^u).
\]
\end{thml}

By definition, we have
$W_r^- = w_r. W_r^\circ$.
Let 
\be \label{c}
c:= \begin{pmatrix} w_n \\ & 1\end{pmatrix} w_{n+1} = \begin{pmatrix} & 1_n \\ 1 & \end{pmatrix}.
\ee
Then 
\be \label{RSeq}
Z(s, W_{n+1}^-, W_n^-) = Z(s, c. W_{n+1}^\circ, W_n^\circ).
\ee
In the next two sections, we compute the Rankin--Selberg integral $Z(s, c.W_{n+1}^\circ, W_n^\circ)$ using Hecke operators and prove Theorem \ref{thm:main}.

\section{Integrals of newforms} \label{sec3}

For $g\in G_r$, the Hecke operator 
\[
T_g: = \frac{1}{{\rm vol}(K_{1,r})}\int_{K_{1,r} g K_{1, r}} \Pi_r(h) \od\!h
\]
acts on the one-dimensional space $\Pi_r^{K_{1,r}}$, and we denote its eigenvalue by $\lambda_g$. 
Put 
\be \label{ab}
a_r: = \begin{pmatrix} \varpi 1_{r-1} \\ & 1\end{pmatrix},\qquad  b_r: = \begin{pmatrix} 1_{r-1} \\ & \varpi\end{pmatrix}.
\ee
By a direct calculation, 
\[
K_{1,r} a_r K_{1,r} = \bigsqcup_{x\in \mathbb{F}_q^{r-1}} u(x) a_r K_{1,r},
\]
hence
\be \label{hecke}
T_{a_r}(v) = \sum_{x\in \mathbb{F}_q^{r-1}} u(x)a_{r+1}. v,\qquad v\in \Pi_r^{K_{1,r}}.
\ee

\begin{leml} \label{lem:eigen}
We have 
\[
\lambda_{a_{n+1}} = q^{\frac{n-1}{2}} \chi^{-1}(\varpi), \quad \lambda_{a_n} = q^{\frac{n-2}{2}} \xi^{-1}(\varpi)\quad \text{and}\quad \lambda_{b_n} = q^{\frac{n-2}{2}} \xi(\varpi).
\]
\end{leml}

\begin{proof}
By \cite[Theorem 3.1]{Matringe}, \eqref{unr} and that the central character of  $\Pi_{n+1}$ is trivial, we have 
\[
W_{n+1}^\circ(a_{n+1}) = q^{-\frac{n}{2}} \cdot q^{-\frac{1}{2}}\chi(\varpi) \cdot \prod^{n-1}_{i=1}\chi_i(\varpi) = q^{-\frac{n+1}{2}}\chi^{-1}(\varpi). 
\]
Then by \eqref{hecke}, 
\[
\lambda_{a_{n+1}} =T_{a_{n+1}}(W_{n+1}^\circ)(1_{n+1}) =  q^n \cdot W^\circ_{n+1}(a_{n+1}) = q^{\frac{n-1}{2}}\chi^{-1}(\varpi). 
\]
Similarly, we have the formula for $\lambda_{a_n}$. Since $b_n = \varpi\cdot a_n^{-1}$ and the central character of $\Pi_n$ is trivial, $T_{b_n}$ is the adjoint of $T_{a_n}$ with respect to an invariant inner product on the unitarizable representation $\Pi_n$, which implies that $\lambda_{b_n} = \overline{\lambda_{a_n}}$. 
\end{proof}

In the sequel, we assume that $\Re(s)$ is sufficiently large such that all the Rankin--Selberg integrals are absolutely convergent. Similar to \eqref{van}, we have 
\[
Z(s, W_{n+1}^\circ, W_n^\circ) = 0. 
\]
Applying $T_{a_{n+1}}$ to $W_{n+1}^\circ$ gives that 
\be \label{sum1}
\sum_{x\in \mathbb{F}_q^n} Z(s, u(x)a_{n+1}. W_{n+1}^\circ, W_n^\circ) =0.
\ee
Write 
\[
x=(x_1,x_2,\ldots, x_n)\in \mathbb{F}_q^n.
\]
We split the sum \eqref{sum1} according to 
whether $x_n=0$ or not. 

\subsection{The summands for $x_n=0$}

\begin{leml} \label{xn=0}
We have 
\[
\sum_{y\in \mathbb{F}_q^{n-1}}Z(s, u(y,0)a_{n+1}. W_{n+1}^\circ, W_n^\circ) = q^{n(s-\frac{1}{2})+n-1} Z(s, W_{n+1}^{\rm av}, W_n^\circ). 
\]
\end{leml}

\begin{proof}
By the $K_{1,n+1}$-invariance of $W_{n+1}^\circ$, we have
\[
\begin{aligned}
Z(s, u(y,0)a_{n+1}. W_{n+1}^\circ, W_n^\circ)\, & = Z(s,  a_{n+1} u(\varpi^{-1}y, 0). W_{n+1}^\circ, W_n^\circ)  \\
& = Z(s,  a_{n+1} u(\varpi^{-1}y, 1). W_{n+1}^\circ, W_n^\circ)  \\
& = q^{n(s-\frac{1}{2})} Z(s,   u(\varpi^{-1}y, 1). W_{n+1}^\circ, W_n^\circ),
\end{aligned}
\]
where the last equality follows from a change of variable $g\mapsto \varpi^{-1} g$ in the Rankin--Selberg integral \eqref{RS}. This implies the lemma.
\end{proof}

\subsection{The summands for $x_n\neq 0$}  Let $e_1, e_2,\ldots, e_n$ be the standard basis of $F^n$. 

\begin{leml}\label{xnneq0}
We have 
\[
\begin{aligned}
Z(s, u(x)a_{n+1}.W_{n+1}^\circ, W_n^\circ) & = Z(s, u(e_n)a_{n+1}.W_{n+1}^\circ, W_n^\circ) \\
& = -\frac{q^{n(s-\frac{1}{2})}}{q-1} Z(s, W_{n+1}^{\rm av}, W_n)
\end{aligned}
\]
for every $x\in \mathbb{F}_q^n$ with $x_n\neq 0$.
\end{leml}

\begin{proof}
Write $x= (y, x_n)$ with $y\in \mathbb{F}_q^{n-1}$. It is easy to verify that
\[
\begin{aligned}
u(x) a_{n+1} & = \begin{pmatrix} x_n 1_n &  \\ & 1\end{pmatrix} u(x_n^{-1}y, 1) a_{n+1} \begin{pmatrix} x_n^{-1}1_n \\ & 1\end{pmatrix} \\
&  =  \begin{pmatrix} x_n 1_n & \\ & 1\end{pmatrix} \begin{pmatrix} u(x_n^{-1}y) \\ & 1\end{pmatrix} u(e_n) a_{n+1} \begin{pmatrix} u(-x_n^{-1}y) \\  & 1\end{pmatrix}  \begin{pmatrix} x_n^{-1}1_n \\ & 1\end{pmatrix}.
\end{aligned}
\]
The first equality of the lemma again follows from the $K_{1,n+1}$-invariance of $W_{n+1}^\circ$ and a change of variable in the Rankin-Selberg integral. The second equality follows from the first one, \eqref{sum1} and 
Lemma \ref{xn=0}.
\end{proof}

\section{Integrals of opposite newforms}  \label{sec4}

In this section, we compute $Z(s, c.W_{n+1}^\circ, W_n^\circ)$ and prove Theorem \ref{thm:main}, where $c$ is given by \eqref{c}.  By \eqref{hecke}, 
\be \label{sum2}
\lambda_{a_{n+1}} Z(s, c.W_{n+1}^\circ, W_n^\circ) = \sum_{x\in \mathbb{F}_q^n}Z(s, c \, u(x)a_{n+1}. W_{n+1}^\circ, W_n^\circ).
\ee
We split the sum according to $x_1=0$ or not.

\subsection{The summands for $x_1=0$}

\begin{leml} \label{x1=0}
We have 
\[
\sum_{z\in \mathbb{F}_q^{n-1}} Z(s, c\, u(0,z) a_{n+1}.W_{n+1}^\circ, W_n^\circ) = q^{-s+\frac{1}{2}}\lambda_{b_n} Z(s, c. W_{n+1}^\circ, W_n^\circ).
\]
\end{leml}

\begin{proof}
We have 
\[
c \, u(0,z)a_{n+1} = \begin{pmatrix} u(z) \\ & 1\end{pmatrix} c\, a_{n+1} = \varpi \cdot  \begin{pmatrix} u(z) \\ & 1\end{pmatrix} \begin{pmatrix} b_n^{-1} \\ & 1\end{pmatrix} c,
\]
where $b_n$ is given by \eqref{ab}. 
By the trivial central character assumption, 
\[
\begin{aligned}
 & \sum_{z\in \mathbb{F}_q^{n-1}} Z(s, c\, u(0,z) a_{n+1}.W_{n+1}^\circ, W_n^\circ)  \\
 =\, &  \frac{1}{{\rm vol}(K_{1,n})} \int_{K_{1,n} b_n^{-1} K_{1,n} } Z\left(s, \begin{pmatrix} h \\ & 1\end{pmatrix}c.W_{n+1}^\circ, W_n^\circ\right)\od\! h \\
 =\, & q^{-s+\frac{1}{2}} Z\left(s, c.W_{n+1}^\circ, T_{b_n}(W_n^\circ)\right) \\
 =\, & q^{-s+\frac{1}{2}}\lambda_{b_n} Z(s, c. W_{n+1}^\circ, W_n^\circ),
\end{aligned} 
\]
where the second last equality follows from a change of variable $g\mapsto gh^{-1}$ in the Rankin--Selberg integral. 
\end{proof}

\subsection{The summands for $x_1\neq 0$}

\begin{leml} \label{x1neq0}
We have 
\[
Z(s, c\, u(x) a_{n+1}.W_{n+1}^\circ, W_n^\circ) = Z(s, u(e_n)a_{n+1}.W_{n+1}^\circ, W_n^\circ)
\]
for every $x\in \mathbb{F}_q^n$ with $x_1\neq 0$. 
\end{leml}

\begin{proof}
Write $x= (x_1, z)$ with $z\in \mathbb{F}_q^{n-1}$. Note that 
\[
c = \begin{pmatrix} X & e_n^{\rm t} \\ e_1 & 0\end{pmatrix},\quad \text{where}\quad X := \left(\begin{smallmatrix}  0 & 1 \\ & 0 & 1 \\  & & \ddots & \ddots \\ & & & 0 & 1 \\ & & & & 0\end{smallmatrix}\right)_{n\times n}.
\]
It is straightforward to verify the Iwasawa decomposition 
\[
c\, u(x) a_{n+1} = \underbracket{\begin{pmatrix} h_x \\ & 1\end{pmatrix} u(e_n) a_{n+1}}_{\in B_{n+1}} k_x,
\]
where
\[
h_x:=\begin{pmatrix} 1_{n-1} & x_1^{-1}z^{\rm t}  \\ 0 & x_1^{-1}\end{pmatrix}\in K_n \quad \text{and}\quad k_x:= \begin{pmatrix} h_x^{-1} X - e_n^{\rm t} e_1 & 0 \\ \varpi e_1 & x_1\end{pmatrix}\in \frak o^\times K_{1,n+1}.
\]
Then $k_x.W_{n+1}^\circ = W_{n+1}^\circ$ by the trivial central character assumption, hence 
\[
Z(s, c\, u(x) a_{n+1}.W_{n+1}^\circ, W_n^\circ) = Z(s, u(e_n)a_{n+1}.W_{n+1}^\circ, W_n^\circ)
\]
by a change of variable $g\mapsto gh_x^{-1}$. 
\end{proof}

\subsection{Proof of Theorem \ref{thm:main}} Now we are ready to prove Theorem \ref{thm:main}. Combing previous results, we have 
\[
\begin{aligned}
& (\lambda_{a_{n+1}} - q^{-s+\frac{1}{2}} \lambda_{b_n})Z(s, c.W_{n+1}^\circ, W_n^\circ) & \\
=\, &   (q-1)q^{n-1} Z(s, u(e_n)a_{n+1}.W_{n+1}^\circ, W_n^\circ)  &  (\text{by \eqref{sum2}, Lemmas \ref{x1=0}, \ref{x1neq0}}) \\
=\, &  -q^{n(s-\frac{1}{2})+n-1} Z(s, W_{n+1}^{\rm av}, W_n^\circ)  &  (\text{by Lemma \ref{xnneq0}}) \\
=\, & -\frac{q^{n(s-\frac{1}{2})+n-1}}{\#\mathbb{P}^{n-1}(\mathbb{F}_q)} L(s, \Pi_{n+1}^u\times \Pi_n^u) & (\text{by Theorem \ref{BKL}})\\
=\, & -\frac{q^{n(s-\frac{1}{2})+n-1}(1-q^{-s}\chi(\varpi)\xi(\varpi))}{\#\mathbb{P}^{n-1}(\mathbb{F}_q)} L(s, \Pi_{n+1}\times \Pi_n) & (\text{by \eqref{naive}}).
\end{aligned}
\]
Taking into account the Hecke eigenvalues $\lambda_{a_{n+1}} = q^{\frac{n-1}{2}} \chi^{-1}(\varpi)$ and $\lambda_{b_n} = q^{\frac{n-2}{2}} \xi(\varpi)$ from Lemma \ref{lem:eigen}, it follows that 
\[
Z(s, c.W_{n+1}^\circ, W_n^\circ) = -\frac{\chi(\varpi) q^{ns- \frac{1}{2}}}{ \# \mathbb{P}^{n-1}(\mathbb{F}_q)}
L(s,\Pi_{n+1}\times\Pi_n).
\]
In view of \eqref{RSeq}, this proves Theorem \ref{thm:main}.

\section{Nonvanishing of local relative characters} \label{sec:conj}

As an application of Theorem \ref{thm:main},  in this section we sketch and prove a case of \cite[Conjecture 3.6.10]{DZ} in the  local setting, and we refer to \cite[\S3]{DZ} for more details. 

Assume that $q$ is odd, $F$ is an unramified quadratic extension of $F_0$ with nontrivial Galois automorphism ${\sf c}$ and $\varpi\in F_0$, and that $\psi$ is trivial on $F_0$. Let $\eta: F^\times \to \{\pm1\}$ be the quadratic character associated to $F/F_0$ via local class field theory. 

Suppose that the tempered representation 
$\Pi := \Pi_{n}\otimes \Pi_{n+1}$ of $G':=G_{n}/G_1^{\sf c}\times G_{n+1}/G_1^{\sf c}$ is hermitian with respect to $F/F_0$, that is,
$
\Hom_{G'^{\sf c}}(\Pi, \eta^{n-1}\boxtimes\eta^n)
$
is nonzero.  Let $(V, \pi)$ be the distinguished element in the Vogan $L$-packet of $\Pi$ (\cite{GGP}). Here 
\[
V= (V_n, V_{n+1}= V_n \oplus Fu)
\]
is a relevant pair of $F/F_0$-hermitian spaces, and $\pi$ is an irreducible tempered 
representation of the unitary group $G:=U(V_n)\times U(V_{n+1})$ with base change $\Pi$. 
%Let $H':=\Delta G_n$ and $H:=\Delta U(V_n)$, diagonally embedded into $G'$ and $G$ respectively. 

In \cite[\S3.2]{DZ}, a character $I_\Pi(f')$ for $f'\in \mathscr{H}(G')$ (the space of Schwartz measures on $G'$) is defined. In particular, if 
\[
f'= e_{K'}:= \frac{1}{{\rm vol}(K', \od\!g')} {\bf 1}_{K'}\od\!g'
\]
is associated to an open compact subgroup $K'$ of $G'$, then up to normalization  by local factors, $I_\Pi(e_{K'})$ is given by
\be \label{trace}
\sum_W P_1(W)P_2(W^\vee),
\ee
where $\{W\}$ is a basis of $\Pi^{K'}$ and $\{{W^\vee}\}$ is the dual basis of $(\Pi^\vee)^{K'}$. Here $\Pi$ and $\Pi^\vee$ are identified with their Whittaker models with respect to $\psi_n\boxtimes\psi_{n+1}$ and $\bar\psi_n\boxtimes\bar\psi_{n+1}$ respectively, and $P_1$ and $P_2$ are the Rankin-Selberg and Flicker-Rallis integrals respectively:  
\[
\begin{aligned}
P_1(W):& = \int_{N_n\bs G_n} W\left(h, \begin{pmatrix} h \\ & 1\end{pmatrix}\right)\od\!h, \\
P_2(W^\vee):& =\int_{N_{n-1}^{\sf c}\bs G_{n-1}^{\sf c}\times N_{n}^{\sf c}\bs G_{n}^{\sf c} } W^\vee \left( \begin{pmatrix} \epsilon_{n-1}'(\tau)h_1 \\ & 1 \end{pmatrix}, \begin{pmatrix} \epsilon_{n}'(\tau) h_2 \\ & 1\end{pmatrix}\right) \\
&\qquad\qquad \cdot \eta(\det h_1)^{n-1} 
\eta(\det h_2)^n \od\!h_1\!\od\!h_2,
\end{aligned}
\]
where $\tau\in F^\times$ is a trace zero element, and 
$
\epsilon_r'(\tau) := \diag(\tau^{r+\epsilon_r-1}, \tau^{r+\epsilon_r-2},\ldots, \tau^{\epsilon_r-1})\in G_r
$
with $\epsilon_r \in \{0,1\}$ having the same parity as $r$. See \cite[\S2]{DZ} for the choice of various measures, and \cite[\S3.2.2]{DZ} for the pairing between $\Pi$ and $\Pi^\vee$.

In \cite[\S3.4]{DZ}, a character $J_\pi(f)$ for $f\in \mathscr{H}(G)$ is also defined, which up to normalization by local factors is given by
\[
\int_{H} {\rm Tr} (\pi(h)\pi(f))\od\! h,
\]
where $H$ is the image of the diagonal embedding of $U(V_n)$ into $G$.

Let $\epsilon\in \{0,1\}$ have the same parity as $v(e)$, where $v$ is the normalized valuation on $F_0$, and $e:=\langle u, u\rangle$ with $\langle\cdot,\cdot\rangle$ the hermitian form on $V_{n+1}$. 
An $\frak o$-lattice $\Lambda_r\subset V_r$ is called a vertex lattice if 
\[
\Lambda_r \subset \Lambda_r^\vee \subset \varpi^{-1}\Lambda_r,
\]
where $\Lambda_r^\vee$ is the dual lattice with respect to a hermitian form on $V_r$. In this case $t(\Lambda_r):=\dim_{\mathbb{F}_q}(\Lambda_r^\vee/ \Lambda_r)$ is called the type of $\Lambda_r$. 
If $\Lambda_n$ is a vertex lattice of type $t$, then 
\[
\Lambda_{n+1}: = \Lambda_n\oplus \frak o \varpi^{-\lfloor v(e)/2\rfloor} u
\]
is a vertex lattice of type $t+\epsilon$, and the stabilizer $K$ of $\Lambda:=\Lambda_n\times \Lambda_{n+1}$ in $G$ is called a vertex paraphoric subgroup of type $(t, t+\epsilon)$. 
Associated to $K$ is the stabilizer $K' \subset G'$ of the lattice chain $(\Lambda \subset \Lambda^\vee)$. 

Then Corollary \ref{cor} amounts to the following result, which proves a case of \cite[Conjecture 3.6.10]{DZ}.

\begin{thml}\label{DZ}
Let the notation and assumptions be as above. Let $f= e_K$ where $K\subset G$ is a vertex parahoric subgroup of type $(1,1)$ if $\epsilon=0$, type $(n-1, n)$ if  $\epsilon =1$. 
Let $f' = e_{K'}$ where $K'\subset G'$ is associated to $K$ as above. Then 
\[
I_\Pi(e_{K'}) = J_\pi(e_K)\neq 0.
\]
\end{thml}

\begin{proof}
The vertex parahoric transfer  $I_\Pi(e_{K'}) = J_\pi(e_K)$ is due to \cite[Theorem 1.1]{Z}. It remains to show that $I_\Pi(e_{K'})\neq 0$. 
Since $\Pi_n$ and $\Pi_{n+1}$ have conductor one,  $V$ has Hasse invariant $\epsilon(V)= -1$ and $\pi = \pi_n\otimes \pi_{n+1}$ is the unique almost unramified representation in the Vogan $L$-packet by \cite[Lemma 3.6.7]{DZ} (that is, $\pi_r$ has a non-zero vector invariant under the stabilizer of a vertex lattice of type $1$ or $r-1$, where $r=n, n+1$).

Hence if $\epsilon =0$,  then we may assume that $V_r = F^{r\times 1}$ with hermitian form represented by $\diag(\varpi^{-1}, 1_{r-1})$,  and 
\[
\Lambda_r= \varpi\frak o \oplus \frak o^{r-1}\subset \Lambda_r^\vee = \frak o^r,
\]
so that the stabilizer of $(\Lambda_r \subset \Lambda_r^\vee)$ in $G_r$ is the parahoric subgroup $\CK_{1,r}^-:=\frak o^\times K_{1,r}^-$. If $\epsilon =1$,  rescaling the hermitian form by $\varpi$, we may assume that 
\[
\Lambda_r = \frak o^r \subset \Lambda_r^\ve = \frak o \oplus \varpi^{-1}\frak o^{r-1},
\]
which has the same stabilizer $\CK_{1,r}^-$ in $G_r$. 
Then $K'$ is the image of $\CK_{1,n}^-\times\CK_{1,n+1}^-$ in $G'$. 

In view of \eqref{trace}, to show  $I_\Pi(e_{K'})\neq 0$, it suffices to prove that $P_1(W) \neq 0$ and $P_2(W^\vee)\neq 0$, where $W:=W_n^-\otimes W_{n+1}^-\in \Pi^{K'}$ with dual basis $W^\vee \in (\Pi^\vee)^{K'}$. 
By Theorem \ref{thm:main}, we have $P_1(W)\neq 0$. By \cite[Theorem 6.1]{AM}, we have $P_2(W^{\vee,\circ})\neq 0$, where $W^{\vee, \circ}: =(w_n, w_{n+1}).W^\vee$ is a newform of $\Pi^\vee$. 
Since $P_2\in \Hom_{G'^{\sf c}}(\Pi^\vee, \eta^{n-1}\boxtimes \eta^n)$, it follows that $P_2(W^\vee)\neq 0$ as well. This proves the theorem.
\end{proof}

\begin{remarkl}
In general, assume that $\Pi=\Pi_n\otimes \Pi_{n+1}$ is a tempered hermitian irreducible representation of $G'$ with trivial central character, such that $\Pi_r$ ($r=n, n+1$) has conductor at most one. In view of \cite[Conjecture 3.10, Remark 3.11]{DZ},  \cite[Theorem 1.2]{D} and Theorem \ref{DZ}, the remaining case is that $\Pi_n$ has conductor one and $\Pi_{n+1}$ is unramified, for which some challenging problems remain to be solved.
%a good test function is not known yet.
\end{remarkl}

\subsection*{Acknowledgement} 

We sincerely thank Wei Zhang for suggesting this problem, for explaining its application to \cite[Conjecture 3.6.10]{DZ}, and for his helpful comments on an earlier draft of this paper.

We gratefully acknowledge the AI-related support provided by the Digital Principal Initiative at Zhejiang Lab. We further thank the Evolutionary Ensemble (EvE) framework of Yu and Yang \cite{yu2026eve}. 
An agent based on this framework was used to independently investigate the problem, providing an alternative line of reasoning and serving as an independent check that confirmed and verified our proof.  

The authors reviewed all AI-assisted contributions and take full responsibility for the final content of the manuscript.

The first author has been partially supported by National Natural Science Foundation of China No. 12526208.

%\subsection*{Conflict of interest}  On behalf of all authors, the corresponding author states that there is no conflict of interest.

%\subsection*{Data Availability} Our manuscript has no associated data.


\begin{thebibliography}{99}

\bibitem[AGRS10]{AGRS}
A.~Aizenbud, D.~Gourevitch, S.~Rallis, and G.~Schiffmann,
\emph{Multiplicity one theorems},
Ann. of Math. (2) \textbf{172} (2010), no.~2, 1407--1434.

\bibitem[AM17]{AM}
U. K. Anandavardhananand, and N.Matringe,
{\it Test vectors for local periods},
Forum Math. 29 (2017), no. 6, 1245--1260.

%\bibitem[BBBG24]{BBBG}
%B.~Brubaker, V.~Buciumas, D.~Bump, and H.~P.~A.~Gustafsson,
%\emph{Colored vertex models and Iwahori Whittaker functions},
%Selecta Math. (N.S.) \textbf{30} (2024), article no.~78.

%\bibitem[BHK98]{BHK}
%C.~J.~Bushnell, G.~M.~Henniart, and P.~C.~Kutzko,
%\emph{Local Rankin--Selberg convolutions for $\GL_n$: explicit conductor formula},
%J. Amer. Math. Soc. \textbf{11} (1998), no.~3, 703--730.

\bibitem[BKL20]{BKL}
A. Booker, M. Krishnamurthy, and M. Lee, 
{\it Test vectors for Rankin--Selberg L-functions},
J. Number Theory 209 (2020), 37--48.

\bibitem[D25]{D}
G. Dang,
{\it Local newforms and spherical characters for unitary groups},
J. Lond. Math. Soc. (2) 111 (2025), no. 6, Paper No. e70203, 35 pp.

\bibitem[DZ]{DZ}
D. Disegni, and W. Zhang,
{\it Gan--Gross--Prasad cycles and derivatives of $p$-adic $L$-functions},
\href{https://arxiv.org/abs/2410.08401}{arXiv:2410.08401}.


\bibitem[GGP12]{GGP}
W. T. Gan, B. H. Gross, and D. Prasad, {\it Symplectic local root numbers, central critical L values, and restriction problems in the representation theory of classical groups}, Ast\'erisque {\bf 346} (2012), 1--109.

\bibitem[J12]{J}
H.~Jacquet,
\emph{A correction to ``Conducteur des repr\'esentations du groupe
lin\'eaire''},
Pacific J. Math. \textbf{260} (2012), no.~2, 515--525.

\bibitem[JPSS81]{JPSS81}
H.~Jacquet, I.~I.~Piatetski-Shapiro, and J.~Shalika,
\emph{Conducteur des repr\'esentations du groupe lin\'eaire},
Math. Ann. \textbf{256} (1981), no.~2, 199--214.

\bibitem[JPSS83]{JPSS}
H.~Jacquet, I.~I.~Piatetski-Shapiro, and J.~Shalika,
\emph{Rankin--Selberg convolutions},
Amer. J. Math. \textbf{105} (1983), no.~2, 367--464.

\bibitem[M13]{Matringe}
N.~Matringe,
\emph{Essential Whittaker functions for $GL(n)$},
Doc. Math. \textbf{18} (2013), 1191--1214.

\bibitem[M14]{Miyauchi}
M.~Miyauchi,
\emph{Whittaker functions associated to newforms for $\GL(n)$ over $p$-adic fields},
J. Math. Soc. Japan 66 (2014), no. 1, 17--24.

\bibitem[YY26]{yu2026eve}
Z. Yu, and L. Yang, {\it Evolving Ensemble of Agents}, 
\href{https://arxiv.org/abs/2605.09018v4}{arXiv:2605.09018}.

\bibitem[Z25]{Z}
Z. Zhang,
{\it Maximal parahoric arithmetic transfers, resolutions and modularity},
Duke Math. J. 174 (2025), no. 1, 1--129.

%\bibitem[T14]{T}
%N.~Templier,
%\emph{Large values of modular forms},
%Cambridge J. Math. \textbf{2} (2014), no.~1, 91--116.

%\bibitem[Z80]{Zelevinsky}
%A.~V.~Zelevinsky,
%\emph{Induced representations of reductive $p$-adic groups. II. On
%irreducible representations of $\GL(n)$},
%Ann. Sci. \`Ecole Norm. Sup. (4) \textbf{13} (1980), no.~2, 165--210.

\end{thebibliography}
\end{document}